\documentclass[11pt,twoside]{article}

\usepackage{mathrsfs,amsfonts,amsmath,amssymb}
\usepackage{latexsym}
\usepackage{comment} 
\newtheorem{satz}{Theorem}
\newtheorem{proposition}[satz]{Proposition}
\newtheorem{theorem}[satz]{Theorem}
\newtheorem{lemma}[satz]{Lemma}
\newtheorem{definition}[satz]{Definition}
\newtheorem{corollary}[satz]{Corollary}
\newtheorem{remark}[satz]{Remark}
\newtheorem{example}[satz]{Example}

\def\Z{\mathbb {Z}}
\def\F{\mathbb {F}}

\def\a{\alpha}

\def\C{\mathbb{C}}

\def\d{\delta}

\def\({\big (}
\def\){\big )}

\def\G{\Gamma}

\def\le{\leqslant}
\def\ge{\geqslant}
\def\_phi{\varphi}
\def\eps{\varepsilon}

\def\Gr{{\mathbf G}}
\def\FF{\widehat}

\def\Spec{{\rm Spec\,}}

\def\la{\lambda}
\def\D{\Delta}

\def\T{\mathsf{T}}

\def\C{\mathbb{C}}

\def\tr{\mathrm{tr}}

\def\Cay{{\rm Cay}}
\def\Bohr{{\rm Bohr}}

\author{Shkredov I.D.}
\title{
	On the spectral gap and the diameter of Cayley graphs
\footnote{This work is supported by the Russian Science Foundation under grant 19--11--00001.}
}
\date{}
\begin{document}
	\maketitle


\begin{center}
	Annotation.
\end{center}

{\it \small
	We obtain a new bound  connecting the first non--trivial eigenvalue of the Laplace operator of a graph and the diameter of the 
	graph, 
	which is effective for graphs with small diameter or for graphs, having the number of maximal paths 
	comparable 
	to the expectation. 
}
\\

\section{Introduction}
\label{sec:introduction}

Expander graphs were first introduced by Bassalygo and Pinsker \cite{BP}, and their existence first proved by Pinsker \cite{Pinsker_expander} (also, see \cite{Margulis}).
The property of a graph of being an expander is significant in many of mathematical and computational  contexts, see, e.g., \cite{Kowalski_exp}, \cite{Lubotzky}, \cite{Saloff}.
It is well--known that the expansion property of a graph is controlled by the spectral gap of the Laplace operator $\Delta$, namely, by the first non--trivial eigenvalue $\la_1$ of $\Delta$, see \cite{Lubotzky} (all required definitions 
can be found in Section \ref{sec:diameter} below). 
In this paper we study the connection of $\la_1$ and the diameter of a graph
and
we concentrate on {\it Cayley graphs} (although some generalizations are possible as well, see Theorem \ref{t:basis_graph}).
In \cite{DS-C} the following result was obtained (also, see \cite[Corollary 3.2.7]{Saloff}).

\begin{theorem}
	Let 
	$\Gr$ be a finite group. 
	Let $S\subseteq \Gr$ be a set and $d$ be the diameter of its Cayley graph $\Cay (S)$. 
	Then 
	\[
	\la_1 (\Cay (S)) \ge \frac{1}{2d^2 |S|} \,.
	\]
	\label{t:DS-C}
\end{theorem}

Now we formulate our first main result. 

\begin{theorem}
	Let $\Gr$ be a finite group. 
	Let $S\subseteq \Gr$  be a set and $d$ be the diameter of its Cayley graph  $\Cay (S)$. 
	Then 
	\[
	\la_1 (\Cay (S)) \ge \frac{|\Gr|}{d |S|^d} \,.
	\]
	\label{t:basis}
\end{theorem} 

A set $S \subseteq \Gr$ is called a {\it basis} of order $d$ if $S^d = \Gr$. 
It follows that 
Theorem \ref{t:basis} is better than Theorem \ref{t:DS-C} in the case when a basis $S$ of order $d$ 
satisfies 
$|S|^{d-1} < 2d |\Gr|$. 
In particular, our result is better for all possible  $S$ in the case $d=2$.
On the other hand, 
if $d$ is the diameter of $\Cay (S)$, then $|S|^d \ge |\Gr|$.
Thus our result is better than Theorem \ref{t:DS-C}  for "economical"\, basis $S$, i.e., 
in the case when  $\Gr$ 
has no elements require a lot of multiplications of $S$ to be represented. 
For example, assuming the condition $|S|^d \ll_d |\Gr|$, we have 
$\la_1 (\Cay (S)) \gg_d 1$.  
The same bound takes place if the number of representations of any $x\in \Gr$ as $x=s_1 \dots s_d$, $s_j \in S$ is 
$\Omega (|S|^d/|\Gr|)$. 
Other examples of effective using of Theorem \ref{t:basis} 
are contained 
in Remark \ref{r:other_g} and in Section \ref{sec:examples} below. 
Here we show that our new bound for the gap of the Laplace operator allows to say something 
new 
 on  non--commutative sets 
having no solutions of linear equations, Sidon sets,  as well as about the famous Erd\H{o}s--Tur\'an conjecture.

Actually, the methods from \cite[Chapter 3]{Saloff} are rather general and one can obtain an analogue of Theorem \ref{t:DS-C} 
for almost arbitrary graphs.
In this direction we prove

\begin{theorem}
	Let $G = G(V,E)$ be a finite graph with the valency $\mathcal{V}$ and the diameter $d$. 
	Then 
	\[
	\la_1 (G) \ge 
	\frac{|V|}{d\mathcal{V}^d} \,.
	\]
	\label{t:basis_graph}
\end{theorem}

In Sections \ref{sec:Z_N}, \ref{sec:non-abelian} we concentrate on the case of Cayley graphs and obtain a characterisation of the spectral gap in terms of the intersection of our set $S$ with arithmetic progressions and (non--abelian) Bohr sets.
Let us formulate a result from these Sections (see Corollaries \ref{c:basis_ab}, \ref{c:basis_nab}). 

\begin{theorem}
	Let $\Gr$ be a finite group, $\eps \in (0,1)$ be a real number, $d \ge 2$ be an integer and $B, \Omega \subseteq \Gr$, $|\Omega| = (1-\eps) |\Gr|$ 
	be sets such that any element of $\Gr \setminus \Omega$ can be represented as a product of $d$ elements of $B B^{-1}$ or $B^{-1} B$ in at least $g$ ways.
	Suppose that $\Gr$ has no  normal proper  subgroups of index at most $2/\eps$. 
	Then 
	\begin{equation}\label{f:basis_nab+_intr}
	\la_1 (\Cay (B)) \ge \frac{g \eps^{\log_{3/2} 3} |\Gr|}{16d^2 |B|^{2d}}  \,.
	\end{equation}	
	\label{t:basis_nab_intr}
\end{theorem}

In the abelian case the dependence on the parameters in \eqref{f:basis_nab+_intr} is better, see Corollary \ref{c:basis_ab} below. 
Thus Theorem \ref{t:basis_nab_intr} shows that in the case of Cayley graphs one can has a relatively large exceptional set $\Omega$ and nevertheless a rather good lower bound for $\la_1 (\Cay (B))$.
Finally, in Appendix we 
collect 
some simple properties of non--abelian  Bohr sets.


\section{Definitions} 
\label{sec:definitions}

Here and throughout this paper $\Gr$ is a finite group with the identity $e$.
Given two sets $A,B\subset \Gr$, define  the \textit{product set}  of $A$ and $B$ as 
$$AB:=\{ab ~:~ a\in{A},\,b\in{B}\}\,.$$
In a similar way we define the higher product sets, e.g., $A^3$ is $AAA$. 
Let $A^{-1} := \{a^{-1} ~:~ a\in A \}$. 
Having an element $g\in \Gr$ and a positive integer $k$ we write $g^{1/k}$ for the set $\{x \in \Gr ~:~ x^k = g \}$. 
Further, if $A\subseteq \Gr$ is a set, then $A^{1/k}$ equals $\{ a^{1/k} ~:~ a\in A\}$.  
In this paper we use the same letter to denote a set $A\subseteq \Gr$ and  its characteristic function $A: \Gr \to \{0,1 \}$. 
Given a function $f: \Gr \to \C$, we write $\langle f \rangle$ for $\sum_{x \in \Gr} f(x)$.  

Now 
we recall some notions and simple facts from the representation theory, see, e.g., \cite{Naimark} or \cite{Serr_representations}.
For a finite group $\Gr$ let $\FF{\Gr}$ be the set of all irreducible unitary representations of $\Gr$. 
It is well--known that size of $\FF{\Gr}$ coincides with  the number of all conjugate classes of $\Gr$.  
For $\rho \in \FF{\Gr}$ denote by $d_\rho$ the dimension of this representation. 
By $d_{\min} (\Gr)$ denote 
the quantity 
$\min_{\rho \neq 1} d_\rho$.  
We write $\langle \cdot, \cdot \rangle$ for the corresponding  Hilbert--Schmidt scalar product 
$\langle A, B \rangle = \langle A, B \rangle_{HS}:= \tr (AB^*)$, where $A,B$ are any two matrices of the same sizes. 
Put $\| A\|_{HS} = \sqrt{\langle A, A \rangle}$.
Clearly, $\langle \rho(g) A, \rho(g) B \rangle = \langle A, B \rangle$ and $\langle AX, Y\rangle = \langle X, A^* Y\rangle$.
Also, we have $\sum_{\rho \in \FF{\Gr}} d^2_\rho = |\Gr|$.

For any function $f:\Gr \to \mathbb{C}$ and $\rho \in \FF{\Gr}$ define the matrix $\FF{f} (\rho)$, which is called the Fourier transform of $f$ at $\rho$ by the formula 
\begin{equation}\label{f:Fourier_representations}
\FF{f} (\rho) = \sum_{g\in \Gr} f(g) \rho (g) \,.
\end{equation}
Then the inverse formula takes place
\begin{equation}\label{f:inverse_representations}
f(g) = \frac{1}{|\Gr|} \sum_{\rho \in \FF{\Gr}} d_\rho \langle \FF{f} (\rho), \rho (g^{-1}) \rangle \,,
\end{equation}
and the Parseval identity is 
\begin{equation}\label{f:Parseval_representations}
\sum_{g\in \Gr} |f(g)|^2 = \frac{1}{|\Gr|} \sum_{\rho \in \FF{\Gr}} d_\rho \| \FF{f} (\rho) \|^2_{HS} \,.
\end{equation}
The main property of the Fourier transform is the convolution formula 
\begin{equation}\label{f:convolution_representations}
\FF{f*g} (\rho) = \FF{f} (\rho) \FF{g} (\rho) \,,
\end{equation}
where the convolution of two functions $f,g : \Gr \to \mathbb{C}$ is defined as 
\[
(f*g) (x) = \sum_{y\in \Gr} f(y) g(y^{-1}x) \,.
\]
Given a function $f : \Gr \to \mathbb{C}$ and a positive integer $k$, we write  $f^{(k)} = (f^{(k-1)} * f)$ for the $k$th convolution of $f$.  
Finally, it is easy to check that for any matrices $A,B$ one has $\| AB\|_{HS} \le \| A\| \| B\|_{HS}$ and $\| A\| \le \| A \|_{HS}$, where $\| \cdot \|$ is the operator $l^2$--norm  of $A$, that is just 
the maximal singular value of $A$.  
In particular, it shows that $\| \cdot \|_{HS}$ is indeed a matrix norm. 
Also, giving a set $S\subseteq \Gr$ we denote $\min_{\rho \in \FF{\Gr},\, \rho \neq 1} \| \FF{S} (\rho) \|$ as $\| S\|$.

The signs $\ll$ and $\gg$ are the usual Vinogradov symbols.
All logarithms are to base $2$.

\section{On the diameter of Cayley graphs}
\label{sec:diameter}

Let $S\subseteq \Gr$ be a set and let  $\Cay (S)$ be the correspondent  {\it Cayley graph} of $S$ defined as $\Cay(S) = (V,E)$ with the vertex set $V=\Gr$ and the set of edges 
$$ E = \{ (g,gs) ~:~ g\in \Gr,\, s\in S\} \,.$$ 
Clearly, $\Cay(S)$ is a regular graph and its diameter equals minimal $d$ such that $S^d = \Gr$. 
As usual we consider the (oriented) {\it Laplace operator} of $\Cay (S)$ defined for an arbitrary function $f:\Gr \to \C$ as 
\begin{equation}\label{f:Delta_f}
	(\Delta f)(x) = f(x) - |S|^{-1} \sum_{s\in S} f(xs) \,.
\end{equation} 
In other words, the matrix of $\Delta$ is $I-|S|^{-1} M(x,y)$, where $I$ is the identity operator and  $M(x,y)$ is the  adjacency matrix  of the graph $\Cay (S)$
(the {\it Markov operator} of $\Cay (S)$), $M(x,y) = S(x^{-1} y)$. 
Actually, formula \eqref{f:Delta_f} 
defines 
an operator with an arbitrary function $F(x)$ instead of $S(x)$ if one replaces $|S|$ by $\| F\|_1$. 
Further the Laplace operator has the spectrum 
$$0 = \lambda_0 (\Cay (S)) \le \lambda_1 (\Cay (S)) \le |\lambda_2 (\Cay (S))| \le \dots \le |\lambda_{|\Gr|-1} (\Cay (S))|$$ 
and there is a variational description of $\lambda_1 (\Cay (S))$, namely, 
$$
	\lambda_1 (\Cay (S)) = \min_{\langle f \rangle = 0,\, \| f\|_2=1} \langle \Delta f, f \rangle \,.  
$$
The quantity $\lambda_1$ is hugely connected with the expansion properties of the considered graph, see, e.g., \cite{Lubotzky}. 
Below  we write $\la_j$ for $\la_j (\Cay (S))$. 
Also, we will consider the correspondent eigenvalues $0 = \lambda^*_0 \le \lambda^*_1 \le \lambda^*_2 \le \dots \le \lambda^*_{|\Gr|-1}$ of the operator 
$I-|S|^{-2} MM^*$ (one can think about these numbers as squares of "singular"\, values of $\D$).

The same can be defined for an arbitrary graph $G=G(V,E)$, see \cite{Lubotzky}, namely, assuming for simplicity that the valency of $G$ is a constant, say, $\mathcal{V}$, we write   
\begin{equation}\label{f:Delta_f_any_graph}
	(\Delta f)(x) = f(x) - \mathcal{V}^{-1} \sum_{(x,y) \in E} f(y) \,.
\end{equation}


\bigskip

The spectrum of Cayley graph $\Cay(S)$ is closely connected with the Fourier transform of the characteristic function of $S$. 
For example, it is well--known, see 
\cite{SX} or \cite[Proposition 6.2.4]{Kowalski_exp}  that the multiplicity of any $\la_j$, $j\neq 0$ is at least 
$d_{\min} (\Gr)$ because each eigenspace of the Markov operator $M$ is a subrepresentation of  the regular representation. 
We collect a series of  further required simple results on the spectrum of $\Cay (S)$ in  the following

\begin{lemma}
	Let $\Gr$ be a finite group, let $S\subseteq \Gr$ be a set. 
	Then $1-\lambda_j$, $1-\lambda^*_j$ belong to the spectra of matrices $|S|^{-1} \FF{S} (\rho)$, $|S|^{-2} \FF{S} (\rho) \FF{S} (\rho)^*$, correspondingly, where $\rho$ runs over $\FF{\Gr}$.  
	Further $\la_1 \ge 1 - |S|^{-1} \| S \|$ and 
\begin{equation}\label{f:Laplace&representations_la^*}
	\la_1^* = 1-|S|^{-2} \| S\|^2 \,.
\end{equation}
\label{l:Laplace&representations}
\end{lemma}
\begin{proof}
	Let $f'(x) = f(x^{-1})$.  
	The required inclusion follows from the formula $(\Delta f) (x) = f(x) - |S|^{-1} (S * f')(x^{-1})$ and similar for
	$I-|S|^{-2} MM^*$. 
	Let $f$ be an eigenfunction of $\Delta$, that is $(\Delta f) (x) = \mu f(x)$, $\mu \in \mathbb{C}$.  
	Taking the Fourier transform, we derive
\[
	\mu \FF{f} (\rho) = \FF{f} (\rho) - |S|^{-1} \FF{f} (\rho) \FF{S} (\rho)^* \,.
\]
	In other words, 
\[
	0 = \FF{f} (\rho) ((1-\mu)I - |S|^{-1}  \FF{S} (\rho)^*) \,.
\]
	In view of \eqref{f:inverse_representations} we know that there is $\rho \in \FF{\Gr}$ such that $\FF{f} (\rho) \neq 0$ because otherwise $f\equiv 0$. 
	Hence the matrix $(1-\mu)I - |S|^{-1}  \FF{S} (\rho)^*$ cannot be invertible for this $\rho$  and thus $1-\mu$ belongs to the spectrum of $|S|^{-1} \FF{S} (\rho)^*$, which coincides with the spectrum of $|S|^{-1} \FF{S} (\rho)$.

	Further, applying the Cauchy--Schwartz inequality, formula \eqref{f:Parseval_representations} twice,  as well as identity \eqref{f:convolution_representations}, we have for any function $f: \Gr \to \C$, $\| f\|_2 =1$, $\langle f\rangle =0$ that 
\[
	\langle \Delta f, f \rangle = 1- |S|^{-1} \sum_x  (S * f')(x^{-1}) \overline{f(x)} 
		= 
			1- (|S| |\Gr|)^{-1} \sum_{\rho \in \FF{\Gr}} d_\rho \langle \FF{S} (\rho) \FF{f'} (\rho),  \FF{\overline{f}} (\rho) \rangle 
		\ge
\]
\[
	\ge
	1- (|S| |\Gr|)^{-1}  \| S\| \sum_{\rho \in \FF{\Gr}} d_\rho \| \FF{f} (\rho) \|^2 =
	1- \| S\| / |S| \,.
\]

	Finally, to get \eqref{f:Laplace&representations_la^*} we first notice that by the same calculations with $S$ replaced by $S*S^{-1}$, we have $\la^*_1 \ge 1-|S|^{-2} \| S\|^2$. 
	Let us obtain the reverse inequality. 
	Find a certain $\rho \in \FF{\Gr}$, $\rho \neq 1$ and 
	a vector 
	 $\_phi \in \C^{d_\rho}$, $\|\_phi\|_2 = 1$ such that $\| S\|^2 = \langle \FF{S} (\rho) \FF{S} (\rho)^* \_phi, \_phi \rangle$.    
	Using the definition of the Fourier transform, we get 
\begin{equation}\label{f:30.03_1}
	\| S\|^2 = \langle \FF{S} (\rho) \FF{S} (\rho)^* \_phi, \_phi \rangle = \sum_{g\in \Gr} (S*S^{-1}) (g) \langle \rho(g) \_phi, \_phi \rangle 
		:=
			\sum_{g\in \Gr} (S*S^{-1}) (g) F(g) \,.
\end{equation}
	Let us calculate the Fourier transform of $F$.
	Applying the orthogonality relations for any $\pi \in \FF{\Gr}$ see, e.g., \cite[Theorem 1, page 67]{Naimark}, we obtain
\[
	\FF{F} (\pi) = \sum_{i,j} \_phi(i) \overline{\_phi(j)} \sum_{g\in \Gr} \rho(g)_{ij} \pi (g) = \frac{|\Gr|}{d_\rho} \left|\sum_k \_phi(k) \right|^2 \ge 0\,.
\]
	Hence the Fourier transform of $F$ is non--negative and thus $F$ can be written as $f' * f$ for a certain function $f$. 
	Since $\rho \neq 1$, it follows that  $\sum_g F(g) = 0$ (here we have used the orthogonality relations again). 
	It implies
	$\langle f \rangle = 0$. 
	But by the definition of the Laplace operator, we have for any function $f$  that 
\begin{equation}\label{f:TT^*_action}
	\langle M M^* f, f \rangle = |S|^{-2} \sum_{x\in \Gr} (S*S^{-1}) (x) (f' * f) (x) \,.
\end{equation}
	Returning to \eqref{f:30.03_1}, using the fact $\langle f \rangle = 0$ and the variational property of the singular values of $M$, we derive 
\[
	\| S\|^2 = |S|^2 \langle M M^* f, f \rangle \le |S|^2 (1-\la^*_1) 
\]
	or, in other words, $\la^*_1 \le 1-|S|^{-2} \| S\|^2$ 
	as required. 
	$\hfill\Box$
\end{proof}

\bigskip 

The proof of the first main Theorem \ref{t:basis} bases on  an idea from \cite{Mosh_digits_Che}. 
We formulate 
our 
result 
in a slightly more general form.

\begin{theorem}
	Let $\Gr$ be a finite group, $\Omega \subset \Gr$ be a set, let $g$ be a positive real, $d \ge 2$ be an integer,  and let 
	$B\subseteq \Gr$ be a set 
	such that any element of $\Gr \setminus \Omega$ can be represented as a product of $d$ elements of $B$ in at least $g$ ways.
	Then 
\begin{equation}\label{f:basis_g_lambda}
	\la_1 (\Cay (B)) \ge \frac{g|\Gr|}{d (|B|+ g|\Omega|)^d} - \frac{g|\Omega|}{|B|}\,.
\end{equation}
	Suppose that for sets $B_1,B_2 \subseteq \Gr$ one has $(B_1 * B_2)^{(d)} (x) \ge g$ outside $\Omega$. 
	Then 
\begin{equation}\label{f:basis_g_lambda*-}
	\la_1 (\Cay (B_1 *B_2)) \ge 	\frac{g|\Gr|}{d (|B_1||B_2|+g|\Omega|)^{d}} - \frac{g|\Omega|}{|B_1||B_2|} \,.
\end{equation}
	In particular,
\begin{equation}\label{f:basis_g_lambda*}
	\la^*_1 (\Cay (B)) \ge 
	\frac{g|\Gr|}{d (|B|^2+g|\Omega|)^{d}} - \frac{g|\Omega|}{|B|^2} \,.
\end{equation}
\label{t:basis_g}
\end{theorem}
\begin{proof}
	We assume at the beginning that $\Omega = \emptyset$.
	Let $f(x) = f_B (x) = B(x) - |B|/|\Gr|$ be the balanced function of the set $B$.  
	Clearly, we have $\sum_{x\in \Gr} f(x) = 0$, further for an arbitrary  $j$ one has 
	$f^{(j)} (x) = B^{(j)} (x) - |B|^j/|\Gr|$ and hence $\sum_{x\in \Gr} f^{(j)}(x) = 0$.
	For any $k\ge 1$ consider 
\[
		\T_k (f) = \sum_{x\in \Gr} f^{(k)} (x)^2 = \sum_{x\in \Gr} B^{(k)} (x)^2 - \frac{|B|^{2k}}{|\Gr|} \,.
\]
	Using the definition of the Laplace operator and counting the number of cycles of length $2k$ in $\Cay(B)$, we obtain  
\begin{equation}\label{f:28.03_0}
		|\Gr| \T_k (f) 
			=
			|B|^{2k} \sum_{j=0}^{|\Gr|-1} |1-\la_j|^{2k} - |B|^{2k}
			 	= 
		|B|^{2k} \sum_{j=1}^{|\Gr|-1} |1-\la_j|^{2k} \,.
\end{equation}
	Notice that $\T_1(f) < |B|$. 
	We have
\begin{equation}\label{f:28.03_1}
	\T_k (f) = \sum_y \sum_{z_1, z_2} f^{(k-d)} (yz^{-1}_1) f^{(k-d)} (y z^{-1}_2) (B^{(d)} (z_1) - g) (B^{(d)} (z_2) - g) 
\end{equation}
	and by the Cauchy--Schwarz inequality for any $z_1,z_2 \in \Gr$, we obtain 
\begin{equation}\label{f:28.03_2} 
	\sum_y f^{(k-d)} (yz^{-1}_1) f^{(k-d)} (y z^{-1}_2) \le \T_{k-d} (f) \,.
\end{equation}
	For any $x\in \Gr$, we know that $B^{(d)} (x) \ge g$.
	Combining the last inequality with \eqref{f:28.03_1}, \eqref{f:28.03_2}, we derive  
\[
	\T_k (f) \le  \sum_{z_1, z_2} \T_{k-d} (f)  (B^{(d)} (z_1) - g) (B^{(d)} (z_2) - g) 
		= 
			\T_{k-d} (f) (|B|^d - g |\Gr|)^2 \,.
\]
	By induction we see that for any $l$ the following holds 
\[
	\T_{dl+1} (f) \le \T_{1} (f) (|B|^d - g |\Gr|)^{2l} < |B| (|B|^d - g |\Gr|)^{2l} \,.
\]
	Substituting the last bound into \eqref{f:28.03_0}, we obtain 
\[
	(1-\la_1)^{2dl+2} |B|^{2ld+2} |\Gr|^{-1} \le \T_{dl+1} (f) < |B| (|B|^d - g |\Gr|)^{2l} = |B|^{2ld+1} \left( 1- \frac{g|\Gr|}{|B|^d} \right)^{2l} \,.
\]
	Taking $l$ sufficiently large, we get
\[
	1-\la_1 \le \left( 1- \frac{g|\Gr|}{|B|^d} \right)^{1/d} \le 1- \frac{g|\Gr|}{d |B|^d}
\]
	as required.

	Now if $\Omega \neq \emptyset$, then replace the characteristic function of $B$ by $\tilde{B} (x) = B(x) + g\Omega (b^{}x)$, where $b$ is an arbitrary element of $B^{d-1}$.
	Then for any $x\in \Gr$ one has $\tilde{B}^{(d)} (x) \ge g$ and we can apply the arguments above. 
	It gives us 
\[
	\la_1 (\Cay (B)) + \frac{g|\Omega|}{|B|} \ge \la_1 (\Cay (\tilde{B})) \ge \frac{g|\Gr|}{d \|\tilde{B} \|_1^d} = \frac{g|\Gr|}{d (|B|+ g|\Omega|)^d}
\]
	and we 
	have 
	 \eqref{f:basis_g_lambda}.

	It remains to obtain \eqref{f:basis_g_lambda*-} and again we consider firstly the case $\Omega = \emptyset$.
	Let us 
	apply the same arguments with a new function $F(x) = (f_1 * f_2) (x)$ instead of $f$, where $f_1 = f_{B_1}$ and $f_2 = f_{B_2}$.   
	One has 
	$$ 
		\T_1 (F) = \sum_{x\in \Gr} F^2 (x) = \sum_{x\in \Gr} (f_1 * f_2)^2 (x) =  \sum_{x\in \Gr} (B_1 * B_2)^2 (x) - \frac{|B_1|^2 |B_2|^2}{|\Gr|} 
			< |B_1| |B_2| \min\{ |B_1|, |B_2| \} 
	$$
	and we can repeat the arguments above.
	For non--empty $\Omega$ consider the function  $(B_1 * B_2) (x) + g \Omega (bx)$,  where $b$ is an arbitrary element of $(B_1 B_2)^{d-1}$ and apply the arguments as before.
	To obtain \eqref{f:basis_g_lambda*} we just 
	use 
	\eqref{f:basis_g_lambda*-} with $B_1=B$, $B_2 = B^{-1}$ or vice versa. 
	This completes the proof. 
$\hfill\Box$
\end{proof}


\begin{remark}
	Using the well--known  Pl\"unnecke inequality \cite{TV} in the case of  the symmetric (for simplicity) basis $B \subseteq \Gr$ of order $d$
	and an abelian group 
	$\Gr$  one has that for any $A\subseteq \Gr$ the following holds 
	\[
	|A| \cdot \left( \frac{|\Gr|}{|A|} \right)^{1/d}  \le |A| \cdot \left( \frac{|B^d|}{|A|} \right)^{1/d} \le |AB| \,.
	\]
	It shows that $\Cay (B)$ has an expansion property and, in principle,  one can obtain some lower estimates for $\la_1$ in terms of the expansion constant $h(\Cay(B))$,
	see, e.g., \cite[Proposition 3.4.3]{Kowalski_exp}.
	
	Similarly, notice that there is another well--known general bound for the spectrum of a strictly positive matrix $A = (a_{ij})_{i,j=1}^n$, namely, 
	$|\mu_2 (A)| \le \mu_1 (A) \cdot \frac{M-m}{M+m}$, where $M=\max_{i,j} a_{ij}$, $m=\min_{i,j} a_{ij}$ and $\mu_1 (A) \ge |\mu_2 (A)| \ge \dots $ are eigenvalues of the matrix $A$. 
	
	Nevertheless, Theorem \ref{t:DS-C} and our Theorem \ref{t:basis_g} give better bounds than  both considered estimates. 
\end{remark}


\begin{remark}
	\label{r:other_g}
	If for any $x\in \Gr$ one has $B^{(d)} (x) \ge 1$, then, clearly, for an arbitrary integer $l\ge d$ the following holds $B^{(l)} (x) \ge |B|^{l-d}$.
	Thus one can improve bounds \eqref{f:basis_g_lambda}, \eqref{f:basis_g_lambda*} of Theorem \ref{t:basis_g} taking larger $l$ in the case when we know some better lower estimates  for $B^{(l)} (x)$.  
	
	Now suppose that $B^{(d)} (x) \gg |B|^d/|\Gr|$ for any $x\in \Gr$, that is, the number of the representations is 
	comparable 
	to its expectation.
	Then $\la_1 (\Cay (B)) \gg 1/d$ and hence the bound for $\la_1 (\Cay (B))$  does not depend on $|\Gr|$ and $|B|$.
\end{remark}

Combining Theorem \ref{t:basis_g} and Lemma \ref{l:Laplace&representations}, we obtain

\begin{corollary}
	Let $\Gr$ be a finite group, $g$ be a positive real, $d\ge 2$ be an integer,  and let 
$B\subseteq \Gr$ be a set such that for any $x\in \Gr$ 
one has $(B * B^{-1})^{(d)} (x) \ge g$ or $(B^{-1} *B)^{(d)} (x) \ge g$.
Then for an arbitrary non--trivial representation $\rho$ one has 
\[
	\| \FF{B} (\rho) \| \le |B| \left( 1-\frac{g|\Gr|}{d |B|^{2d}} \right)^{1/2}
	\,.
\]
\label{c:basis_g}	
\end{corollary}

The next Corollary shows that basis properties of a set $B$ imply  the uniform distribution of the product  $B^k$ for large $k$.

\begin{corollary}
	Let $\Gr$ be a finite group,  $g$ be a positive real, $d\ge 2$ be an integer,  and let $B\subseteq \Gr$ be a set such that 
	$(B * B^{-1})^d$ or $(B^{-1} * B)^d$ is at least one on $\Gr$. 
	Suppose that $k$ grows to infinity faster than 
	$$
		\frac{d |B|^{2d}}{|\Gr|} \cdot  \log \left( \frac{|\Gr|}{|B|} \right)
	\,.
	$$
	Then for any $x\in \Gr$ one has 
\begin{equation}\label{f:UD_basis}	
	B^{(k)} (x) = \frac{|B|^k}{|\Gr|} (1+o(1)) \,.
\end{equation}
\label{c:UD_basis}	
\end{corollary}
\begin{proof} 
	Without loosing of the generality, we consider the case $B*B^{-1}$.
	Using formula \eqref{f:inverse_representations}, we get 
\begin{equation}\label{tmp:01.04_2}
	B^{(k+2)} (x) = \frac{1}{|\Gr|} \sum_{\rho \in \FF{\Gr}} d_\rho \langle \FF{B}^{k+2} (\rho), \rho (x^{-1}) \rangle
		=
			 \frac{|B|^{k+2}}{|\Gr|}  + \mathcal{E} \,,
\end{equation}
	and our task is to estimate the error term $\mathcal{E}$. 
	By Corollary \ref{c:basis_g}, we have 
	$\| \FF{B} (\rho) \| \le |B| \left( 1-\frac{|\Gr|}{d|B|^{2d}} \right)^{1/2}$ and thus in view of \eqref{f:Parseval_representations}, we get 
\begin{equation}\label{tmp:01.04_3}
	|\mathcal{E}| \le \left(|B| \left( 1 - \frac{|\Gr|}{d|B|^{2d}} \right)^{1/2} \right)^k 
		\cdot 
			\frac{1}{|\Gr|} \sum_{\rho \in \FF{\Gr}} d_\rho \| \FF{B} (\rho) \|^2_{HS} 
	\le
	\left( 1 - \frac{|\Gr|}{d|B|^{2d}} \right)^{k/2} |B|^{k+1} \,.
\end{equation} 
	Comparing \eqref{tmp:01.04_2}, \eqref{tmp:01.04_3}, we obtain the result. 
$\hfill\Box$
\end{proof}

\bigskip 

The same arguments work in the general case in the proof of Theorem \ref{t:basis_graph}. 
We left to the reader the task to insert the exceptional set $\Omega$ in Theorem \ref{t:basis_graph'} below.  

\begin{theorem}
	Let $G = G(V,E)$ be a graph with the valency $\mathcal{V}$.
	Suppose that there are at least $g$ paths of length  $d$ between any two vertices of $G$. 
	Then 
	\begin{equation}\label{f:basis_graph'}
	\la_1 (G) \ge 
	\frac{g|V|}{d\mathcal{V}^d}
	\,.
	\end{equation}
\label{t:basis_graph'}
\end{theorem}
\begin{proof}  
	Let $F(x,y) = I - \mathcal{V}^{-1} M(x,y)$ be the matrix of the operator from \eqref{f:Delta_f_any_graph}, $M$ is the adjacency matrix of the graph $G$  
	and denote by $F^{(j)}$, $M^{(j)}$ the powers of these matrices. 
	Clearly, we have $\sum_{x,y} F(x,y) = 0$ and, moreover, by the definition of the valency, one has
	$\sum_{a} F(a,y) = \sum_{b} F(x,b) = 0$ for any $x$ and $y$. 
	Hence for an arbitrary  $j$ and any $x$, $y$ the following holds 
\begin{equation}\label{f:28.03_-1}
	\sum_{a} F^{(j)} (a,y) = \sum_{b} F^{(j)} (x,b) = 0 \,.
\end{equation}
	For any $k\ge 1$ consider 
	\begin{equation}\label{f:28.03_0'}
	\T_k = \sum_{x,y} F^{(k)} (x,y)^2  = \tr (F^{(k)} (F^{(k)})^*) = \mathcal{V}^{2k} \sum_{j=0}^{|V|-1} |1-\la_j|^{2k} 
	= \mathcal{V}^{2k} \sum_{j=1}^{|V|-1} |1-\la_j|^{2k} \,.
	\end{equation}
	Notice that 
	\begin{equation}\label{f:28.03_0+}
		\T_1 = |V| - 2\mathcal{V}^{-1} \tr (M) + \mathcal{V}^{-2} |E| \le |V| + \mathcal{V}^{-1} |V| \le 2|V| \,.
	\end{equation}
	Using formula \eqref{f:28.03_-1}, we obtain 
	\begin{equation}\label{f:28.03_1'}
	\T_k  = \sum_{x,y} \sum_{a,b} F^{(k-d)} (x,a) F^{(k-d)} (x,b) (M^{(d)} (a,y) - g) (M^{(d)} (b,y) - g)
	\end{equation}
	and by the Cauchy--Schwarz inequality for any $a,b \in V$, we have 
$$
	\sum_{x}	F^{(k-d)} (x,a) F^{(k-d)} (x,b) 
		\le 
			\left( \sum_{x} F^{(k-d)} (x,a)^2 \right)^{1/2}
			\left( \sum_{x} F^{(k-d)} (x,b)^2 \right)^{1/2}
$$
	\begin{equation}\label{f:28.03_2'} 
				:=
					q^{1/2}(a) q^{1/2}(b)  \,.
	\end{equation}
	Clearly, $\| q^{1/2} \|^2_2 = \sum_a q(a) = \T_{k-d}$. 
	For any $x,y$, we know that $M^{(d)} (x,y)  \ge g$.
	Combining the last inequality with \eqref{f:28.03_1'}, \eqref{f:28.03_2'}, we derive  
	\[
	\T_k \le  \sum_{a,b} q^{1/2} (a) q^{1/2} (b)  ((M^{(d)} (M^{(d)})^* ) (a,b) - 2g \mathcal{V}^d  + g^2 |V|)
	\le
	\]
	\[
	\le
	\| M^{(d)} q^{1/2} \|^2_2 + (g^2 |V| - 2g \mathcal{V}^d)  \left( \sum_a q^{1/2} (a) \right)^2
		\le
		(\mathcal{V}^{2d} 
		  + (g^2 |V| - 2g \mathcal{V}^d) |V| ) \T_{k-d}
	\,.
	\]
	By induction and estimate \eqref{f:28.03_0+} we see that 
	\[
	\T_{dl+1}  \le \T_{1}  (\mathcal{V}^{d} - g |V| )^{2l} 
		\le 
		2|V|  (\mathcal{V}^{d} - g |V| )^{2l} \,.
	\]
	Substituting the last bound into \eqref{f:28.03_0'}, we obtain 
	\[
	(1-\la_1)^{2dl+2} \mathcal{V}^{2ld+2}  \le \T_{dl+1} (f) 
	\le  
	2|V| (\mathcal{V}^{d} - g |V| )^{2l}
	=
	2\mathcal{V}^{2ld} |V| \left( 1 - \frac{g|V|}{\mathcal{V}^d} \right)^{2l} \,.
	\]
	Taking $l$ sufficiently large, we get
	\begin{equation*}\label{tmp:13.04_1}
	1-\la_1 \le 
	\left( 1- \frac{g|V|}{\mathcal{V}^d} \right)^{1/d} \le 
		1 - \frac{g|V|}{d\mathcal{V}^d} 
	\end{equation*}
	as required. 
%
%
$\hfill\Box$
\end{proof}

\section{On $\Z/N\Z$--case}
\label{sec:Z_N}

Now we consider the  case of an abelian group $\Gr$ and for simplicity we often take $\Gr$ equals $\Z/N\Z$ with a prime $N$ (bounds for spectral gaps of Cayley graphs in general abelian groups can be found in \cite{Sanders_Ab_lectures}, say). 
In this case we show that results of the previous Section can be obtained via another tool (namely, see Theorem \ref{t:Lev_1-eps} below) and moreover one can characterise the existence of the spectral gap in combinatorial terms.

It is easy to see (or consult Lemma \ref{l:Laplace&representations}) 
that in the abelian case for any set $S \subseteq \Gr$ we have the identity  $\la_1 (\Cay (S)) = 1-|S|^{-1} \| S\|$.  
In other words, for any non--trivial character $\chi$ 
\begin{equation}\label{f:la1_abelian}
	\left|\sum_{s\in S} \chi(s) \right| \le (1-\la_1 (\Cay (S))) |S| 
\end{equation}
and the estimate is attained for a certain $\chi$. 
Thus the estimation of the exponential sums and finding non--trivial upper bounds for the quantity $\la_1$ is the same problem for abelian $\Gr$.

\bigskip 

In this Section our basic tool is \cite[Theorem 1]{Lev_1-eps}.

\begin{theorem}
	Let $A\subseteq \Z/N\Z$ be a set, $\eps \in (0,1)$, $\delta \in (0,1/2)$ be real numbers and $|\FF{A} (1)| \ge (1-2\eps (1-\cos \pi \delta)) |A|$. 
	Then there is $a\in \Z/N\Z$ and $l< \delta N$ such that 
\[
	|A\setminus [a,a+l]| < \eps |A| \,.
\]
\label{t:Lev_1-eps}
\end{theorem}

Given a positive integer $d$, a set $P \subseteq \Gr$ and a 
non--negative 
function $f$ on $\Gr$ denote by 
\begin{equation}\label{f:sigma_d} 
	\sigma^{(d)}_P (f) := \| f\|_1^{-d} \sum_{x \in P} f^{(d)} (x) \le 1 \,.
\end{equation}
We characterise the spectral gap of $\Cay (B)$ in terms of purely combinatorial 
quantity 
\eqref{f:sigma_d}.

\begin{theorem}
	Let $N$ be a prime number, $d$ be a positive integer and  $\eps, \delta \in (0,1)$  be real numbers.
	Suppose that  for any arithmetic progression $P$, $|P| \le \delta N$, $\delta < d/2$ one has $\sigma^{(d)}_P (B) \le 1-\alpha$.
	Then 
	$$\la_1 (\Cay (B)) \ge \frac{2\alpha}{d} \left(1 - \cos \frac{\pi \delta}{d} \right) 
	\,.
	$$
	In the opposite direction for any arithmetic progression $P$, $|P| \le \delta N$ one has   
	$\sigma^{(d)}_P (B) \le 1-\alpha$, where $\alpha = ( 1-(1-\la_1 (\Cay (B)))^d - \pi \d )/2$.  
\label{t:Lev_appl}
\end{theorem}
\begin{proof} 
	To obtain the first part of the required result we	apply Theorem \ref{t:Lev_1-eps} with the parameters $\delta/d$ and $\eps = \la_1 /(2(1-\cos \pi \delta/d)))$. 
	In view of  \eqref{f:la1_abelian}, we have the decomposition $B=B_* \bigsqcup E$, where $B_* = B\cap [a,a+l]$, $a\in \Z/N\Z$, $l<\delta N/d$ and $|E|< \eps |B|$. 
	Let $P =[a,a+l]$.
	Then $dP$ is another arithmetic progression of length at most $\delta N$.  
	Further 
\[
	|B|^d 
	= \sum_x B^{(d)} (x) \le 
	\sum_{x} B^{(d)}_* (x) + d |E| |B|^{d-1} 
	<
	\sum_{x\in dP} B^{(d)}_* (x) + \eps d |B|^{d} 
	=
\]
\[ 
	= 
	|B|^d \sigma^{(d)}_{dP} (B) + \eps d |B|^{d} 
	\le 
	|B|^d (1-\alpha + \eps d)
\]
	or, equivalently, 
\[
	\la_1 \ge \frac{2\alpha}{d} \left(1 - \cos \frac{\pi \delta}{d} \right) 
\]
	as required. 

	To get the second part of our Theorem take any arithmetic progression $P$ such that $\sigma^{(d)}_P (B) > 1-\alpha$, where $\alpha$ will be chosen later.
	Then for any nonzero $r\in \Z/N\Z$, we have 
\begin{equation}\label{f:31.03_1}
	\FF{B}^d (r) = \sum_x B^{(d)} (x) e^{-2 \pi irx/N} = \sum_{x\in P} B^{(d)} (x) e^{-2 \pi irx/N} + \theta \alpha |B|^d \,,
\end{equation}
	where $|\theta| \le 1$ is a certain number. 
	By the assumption $N$ is a prime number. 
	Shifting and choosing $r$ in appropriate way, one can assume that $r=1$ and $P$ is a symmetric progression with the step one, i.e., 
	$P = \{ x\in \Z/N\Z ~:~ |x|\le \delta N/2 \}$.
	Returning to \eqref{f:31.03_1} and applying formula \eqref{f:la1_abelian} to estimate the left--hand side of \eqref{f:31.03_1}, we obtain 
\[
	(1-\alpha) |B|^d < \sum_{x\in P} B^{(d)} (x) \le |B|^d ((1-\la_1)^d + \alpha) + \sum_{x\in P} B^{(d)} (x) |e^{-2 \pi ix/N} - 1|
	\le
	|B|^d ((1-\la_1)^d + \alpha + \pi \delta)
\]
	or, in other words, 
\[
	\a \ge 2^{-1} ( 1-(1-\la_1)^d - \pi \d ) \,.
\]	
	This completes the proof. 
$\hfill\Box$
\end{proof}


\bigskip 

Theorem \ref{t:Lev_appl} has a consequence about the Laplace operator of an arbitrary basis of order $d$.

\begin{corollary}
	Let $N$ be a prime number, $d \ge 2$ be an integer and $B, \Omega \subseteq \Z/N\Z$ 
	be sets such that any element of $\Z/N\Z \setminus \Omega$ can be represented as a sum of $d$ elements of $B$ in at least $g\ge 1$ ways.
	Then 
\begin{equation}\label{f:basis_ab}
	\la_1 (\Cay (B)) \ge \frac{g(N - 2|\Omega|)}{d|B|^d}  \left( 1 - \cos \left( \frac{\pi}{2d} \right) \right) \,.
\end{equation}
	If $|\Omega| = (1-\eps)N$, then 
\begin{equation}\label{f:basis_ab+}
\la_1 (\Cay (B)) \ge \frac{\eps gN}{d|B|^d}  \left( 1 - \cos \left( \frac{\eps \pi}{2d} \right) \right) \,.
\end{equation}
\label{c:basis_ab}
\end{corollary}
\begin{proof} 
	Let $\d \in (0,1)$ be a number which we will choose later, let 
	$P$ be an arbitrary arithmetic progression with  $|P| \le \delta N$ and $P^c := (\Z/N\Z) \setminus P$.  
	Since $B^{(d)} (x) \ge g$ for any $x\in \Gr \setminus \Omega$, we see that 
\begin{equation}\label{f:sigma_basis_P^c}
	\sigma^{(d)}_P (B) = 1 - |B|^{-d} \sigma^{(d)}_{P^c} (B) \le 1 - \frac{g(|P^c|-|\Omega|)}{|B|^d} \le 1 - \frac{gN(1-\d) - g|\Omega|}{|B|^d} \,. 
\end{equation}
	Applying Theorem \ref{t:Lev_appl} with $\a = \frac{gN(1-\d) - g|\Omega|}{|B|^d}$ and $\d=1/2$, we derive
\[
	\la_1 (\Cay (B)) \ge \frac{g(N - 2|\Omega|)}{d|B|^d}  \left( 1- \cos \left( \frac{\pi}{2d} \right) \right) 
\]
	as required. 
	To obtain \eqref{f:basis_ab+} use Theorem \ref{t:Lev_appl} with the parameters $\d = \frac{\eps}{2}$ and $\a = \frac{\eps gN}{2|B|^d}$.  
	This completes the proof. 
$\hfill\Box$
\end{proof}

\bigskip 

	Thus the bound of Corollary \ref{c:basis_ab} is comparable with the estimate from Theorem \ref{t:basis}.
	The main advantage of using Theorem \ref{t:Lev_appl} is reformulation of the problem of counting $\la_1$ in terms of 
	purely combinatorial quantity \eqref{f:sigma_d}. 
	Also, the dependence on $\Omega$ in \eqref{f:basis_ab} is better than in Theorem \ref{t:basis_g}.


\begin{example}
	Put $S=\Lambda \bigcup P \subseteq \Z/N\Z$, where  $\Lambda$ is a randomly chosen set such that $2\Lambda = \Z/N\Z$ (or let $2\Lambda$ is close to $\Z/N\Z$, it is not important), $c_1 >0$ is an absolute constant, 
	$|\Lambda| = c_1 \sqrt{N}$
	and $|P| = C \sqrt{N}$ is an arithmetic progression with step one, $C>0$ is a large parameter.  
	One can easily show that the largest non--zero Fourier coefficient of $S$ coincides with the largest non--zero Fourier coefficient of $P$.
	The last is $|P| (1+o(1))$ and hence 
	$$
		\la_1(\Cay (S)) \ge 1 - \frac{|P|(1+o(1))}{|S|} \ge \frac{c_1}{c_1+C}  + o(1) \gg \frac{1}{C} \,.
	$$ 
	On the other hand, Corollary \ref{c:basis_ab} gives us $\la_1 (\Cay (S)) \gg \frac{1}{(c_1+C)^2} \gg \frac{1}{C^2}$.
	Thus for a fixed large $C$ these bounds 
	have comparable quality. 
\end{example}

\section{The general case} 
\label{sec:non-abelian}

In this Section we generalise the results from Section \ref{sec:Z_N} to the  non--abelian case. 
Following \cite[Section 17]{Sanders_A(G)} define the Bohr sets in a (non--abelian) group $\Gr$. 

\begin{definition}
Let $\G$ be a collection of some unitary representations of $\Gr$ and $\delta \in (0,2]$ be a real number.
Put 
\[
	\Bohr (\G,\delta) = \{ g\in \Gr ~:~ \| \gamma(g) - I \| \le \delta\,, \forall \gamma \in \Gamma  \} \,.
\] 
\end{definition}

Clearly, $e\in \Bohr (\G,\delta)$, and  $\Bohr (\G,\delta) = \Bohr^{-1} (\G,\delta) = \Bohr (\G^*,\delta)$. 
Also, notice that  (see, e.g., formula \eqref{f:1-xy} below)
\begin{equation}\label{f:Bohr_sums}	
	\Bohr (\G,\delta_1) \Bohr (\G,\delta_2) \subseteq \Bohr (\G,\delta_1+\delta_2) \,.
\end{equation}
By left/right invariance of $\| \cdot \|$ one can easily show  (or consult \cite[Lemma 4.1]{Sanders_doubling_metrics}) the normality of Bohr sets, i.e., the identity $x \Bohr (\G,\delta) x^{-1} = \Bohr (\G,\delta)$, which holds for any $x\in \Gr$.
If $\G = \{ \rho \}$, 
then we write just $\Bohr (\rho,\delta)$ for $\Bohr (\G,\delta)$ (a lower bound for size of  $\Bohr (\rho,\delta)$ can be found in \cite[Lemma 17.3]{Sanders_A(G)}). 
Further properties of Bohr sets are contained in the Appendix.

\begin{lemma}
	Let $A\subseteq \Gr$ be a set, $\eps, \delta \in (0,1)$ be real numbers.
	Suppose that for a certain 
	unitary representation $\rho$ 
	one has  $\| \FF{A} (\rho) \| \ge (1-\eps) |A|$. 
	Then $\sum_{g\notin \Bohr (\rho,\delta)} (A * A^{-1}) (g) \le \frac{2\eps}{\d}  |A|^2$.
\label{l:vlF} 
\end{lemma}
\begin{proof} 
	By the assumption $\| \FF{A} (\rho) \| \ge (1-\eps) |A|$.
	It means that 
	\[
		\| |A|^2 I - \sum_{g\in \Gr} (A * A^{-1}) (g) (I - \rho (g)) \| = \|\sum_{g\in \Gr} (A * A^{-1}) (g) \rho (g) \| \ge (1-\eps)^2 |A|^2 \,.
	\] 
	For any $g\in \Gr$ each operator $I - \rho (g)$ is normal and non--negatively defined.
	Moreover, the operator 
	$\frac{1}{2} ( (A * A^{-1}) (g) (I - \rho (g)) + (A * A^{-1}) (g^{-1}) (I - \rho (g^{-1})) )$ is hermitian because $(A * A^{-1}) (g^{-1}) = (A * A^{-1}) (g^{})$.
	Hence an arbitrary combination of such operators  with non--negative 
	coefficients  is hermitian and   non--negatively defined as well. 
	It gives 
	\begin{equation}\label{tmp:01.04_1}
	\| \sum_{g\notin \Bohr (\rho,\delta)} (A * A^{-1}) (g^{}) (I - \rho (g)) \| 
		\le \| \sum_{g\in \Gr} (A * A^{-1}) (g^{}) (I - \rho (g)) \| \le (2\eps - \eps^2) |A|^2  
	\end{equation}
	because $\Bohr (\rho,\delta)$ is a symmetric set. 
	Again, for an arbitrary $g\notin \Bohr (\rho,\delta)$ each operator $I - \rho^* (g)$ is normal and positively defined and, moreover, any such 
	operator has all its singular values at least $\delta$	in view of the definition of Bohr sets. 
	Also, $A (g) \ge 0$ for any $g\in \Gr$. 
	Thus by the variational principle we derive from \eqref{tmp:01.04_1} that 
\[
	\d \sum_{g\notin \Bohr (\rho,\delta)} (A * A^{-1}) (g)  \le (2\eps - \eps^2) |A| \le 2\eps |A|^2 
\]
	as required. 
$\hfill\Box$
\end{proof}

\bigskip

Now we are ready to obtain a non--abelian analogue of Theorem \ref{t:Lev_appl}.

\begin{theorem}
	Let $d$ be a positive integer and  $\eps, \delta \in (0,1)$ be real numbers.\\ 
	Suppose that  for any Bohr set $P = \Bohr (\rho, \delta)$, $\rho \neq 1$   one has $\sigma^{(d)}_P (B* B^{-1}) \le 1-\alpha$.
	Then 
	$$\la_1 (\Cay (B)) \ge \frac{\alpha \delta }{2d^2}  
	\,.
	$$
	In the opposite direction for any Bohr set $P = \Bohr (\rho, \delta)$, $\rho \neq 1$  one has   
	$\sigma^{(d)}_P (B* B^{-1}) \le 1-\alpha$, where 
	$$
		\alpha 
			= 
				\frac{1-(1-\la^*_1 (\Cay (B)))^{d} - \d}{2} 
		\,.  
	$$
	\label{t:Lev_appl_non-abelian}
\end{theorem}
\begin{proof}
	By the first part of Lemma \ref{l:Laplace&representations} one has $\|B\| \ge |B|(1-\la_1)$.  
	In other words, for a certain $\rho \neq 1$, we have $\| \FF{B} (\rho) \| \ge |B| (1-\la_1)$. 
		To obtain the first statement of the required result we	apply Lemma \ref{l:vlF}  with the parameters $\delta = \delta/d$ and $\eps = \la_1$. 
	We have the decomposition of the function $f(x) = (B * B^{-1}) (x)$ as $B(x) = f_1 (x) + f_2 (x)$, where the function $f_1$ is supported on the Bohr set 
	$P_* = \Bohr (\rho, \d)$, the function $f_2$ is supported outside $P_*$ and 	$\|f_2 \|_1 \le \frac{2\eps d}{\d}  |B|^2$.  
	Further
	\[
	|B|^{2d} 
	= \sum_x f^{(d)} (x) \le 
	\sum_{x} f^{(d)}_1 (x) + \frac{2d^2 \eps}{\d} |B|^{2d} 
	=
	|B|^{d} \sigma^{(d)}_{P^d_*} (B * B^{-1}) + \frac{2d^2 \eps}{\d} |B|^{2d}
	\le
\]
\[ 
	\le
	|B|^{2d} \left( 1-\alpha + \frac{2d^2 \eps}{\d} \right)
	\]
	or, equivalently, 
	\[
	\la_1 \ge \frac{\alpha \delta }{2d^2} \,.
	\]

	To get the second part of our Theorem take any Bohr set $P = \Bohr (\rho, \delta)$, $\rho \neq 1$ such that $\sigma^{(d)}_P (B* B^{-1}) > 1-\alpha$, where $\alpha$ will be chosen later.
	We have 
	\begin{equation}\label{f:31.03_1'}
	\FF{f}^d (\rho) = \sum_x (B* B^{-1})^{(d)} (x) \rho (x) = \sum_{x\in P} (B* B^{-1})^{(d)} (x) \rho (x) + \theta \alpha |B|^{2d} \,,
	\end{equation}
	where $|\theta| \le 1$ is a certain number. 
	Further  in view of the second part of Lemma \ref{l:Laplace&representations} we can estimate $\| \FF{f}^d\|$ as $(1-\la^*_1)^d |B|^{2d}$.
	It gives 
	\[
	(1-\alpha) |B|^{2d} < \sum_{x\in P} (B* B^{-1})^{(d)} (x) \le |B|^{2d} ((1-\la^*_1)^{d} + \alpha) + \sum_{x\in P} (B* B^{-1})^{(d)} (x) \|\rho (x) - I \|
	\le
	\]
	\[
	\le
	|B|^{2d} ((1-\la^*_1)^{d} + \alpha + \delta)
	\]
	or, in other words, 
	\[
	\a \ge 2^{-1} ( 1-(1-\la^*_1)^{d} - \d ) \,.
	\]	
	This completes the proof. 
$\hfill\Box$
\end{proof}

\begin{remark}
	Clearly, if for any $x\in \Gr$ and any Bohr set $P$ one can estimate from above the intersections $|B\cap Px|$ or $|B\cap xP|$
	as 
	 $(1-\alpha) |B|$, then for an arbitrary  $d$ 
	the following holds 
	$\sigma^{(d)}_{P} (B*B^{-1}) \le 1-\alpha$. 
\end{remark}

We need in upper bounds for Bohr sets.

\begin{lemma}
	Let $\Gr$ be a finite group and $\rho$ be an irreducible representation, $\rho \neq 1$. 
	Then for 
\begin{eqnarray}\label{f:B/2}
	\delta \le \frac{1}{\sqrt{2}} \left( 1-\frac{1}{d_\rho} \right)^{1/2} \,, \quad d_\rho > 1 
		\quad \quad \mbox{and} \quad \quad  
	\delta \le \frac{\sqrt{3}}{2} \,, \quad d_\rho = 1
\end{eqnarray}
	the following holds 
\begin{equation}\label{f:B/2_concl}
	|\Bohr (\rho, \delta)| \le |\Gr|/2 \,.
\end{equation}
	Moreover, if $\Gr$ has no normal proper  subgroups of index at most $1/\eps$, $\eps \le 1/2$,  then 
\begin{equation}\label{f:B/2_concl2}
	|\Bohr (\rho, \delta_\eps)| \le \eps |\Gr| \,, 
\end{equation}
	where 
\[
	\delta_\eps \le \left( 2-\frac{2}{d_\rho} \right)^{1/2}  \cdot \eps^{\log_{3/2}2} \,, \quad d_\rho > 1 
	\quad \quad \mbox{and} \quad \quad  
	\delta_\eps \le \sqrt{3} \cdot \eps^{\log_{3/2}2} \,, \quad d_\rho = 1 \,.
\]
\label{l:B/2}
\end{lemma}
\begin{proof}
	Take $\d$ as in  \eqref{f:B/2}. 
	If $|\Bohr (\rho, \delta)| > |\Gr|/2$, then $\Bohr (\rho, \delta)^2  = \Gr$ and hence $\Bohr (\rho, 2\delta) = \Gr$. 
	In other words, for any $g\in \Gr$ one has $\| \rho(g) - I \| \le 2\d$.
	But 
\begin{equation}\label{tmp:18.04_1}
	2d_\rho - 2 \tr (\rho(g)) = \| \rho(g) - I \|^2_{HS}  \le d_\rho \| \rho(g) - I \|^2 
\end{equation}
	and, on the other hand, by the orthogonality relations and the  irreducibility of $\rho$ one has 
\[
	\sum_{g\in \Gr} |\tr (\rho(g))|^2 = |\Gr| \,.
\]
	Hence there is $g$ such that $|\tr (\rho(g))| \le 1$ and  in view of \eqref{tmp:18.04_1}, we obtain 
\[
	2d_\rho -2 \le d_\rho (2\d)^2 
\]
	as required. 
	Finally, if $d_\rho =1$, then $\rho$ is just a non--trivial character on $\Gr$ and 
\[
	\max_{g\in \Gr} \| \rho(g) - I \| \ge  \min_{1<k \,|\, |\Gr|} \max_n |e^{2\pi i n/k} - 1| \ge \sqrt{3} \,,
\]
	where the minimum is taken over all divisors of $|\Gr|$.

	It remains to obtain \eqref{f:B/2_concl2}. 
	Suppose that $|\Bohr (\rho, \delta_\eps)| > \eps |\Gr|$. 
	We know that any Bohr set is normal.
	Also, it is well--known (see, e.g., \cite{TV}) that for any set $A \subseteq \Gr$ one has either $|AA| \ge 3|A|/2$ or $AA^{-1}$ is a subgroup of $\Gr$. 
	By the assumption $\Gr$ has no normal proper subgroups of index at most $1/\eps$.
	Thus for an integer $k \ge (1/2\eps)^{\log_{3/2} 2}+1$ one has $\Bohr^k (\rho, \delta_\eps) = \Gr$ and hence $\Bohr (\rho, k\delta_\eps) = \Gr$. 
	It follows that $k \delta_\eps$ is greater than $(2-2/d_\rho)^{1/2}$ for $d_\rho >1$ and $\sqrt{3}$ for $d_\rho = 1$.  
	This completes the proof. 
$\hfill\Box$
\end{proof}

\bigskip 

Clearly, estimate \eqref{f:B/2_concl} is tight as the case $\Gr = \F_2^n$ shows. 
Finally, notice a well--known fact that for any $H<\Gr$ one has $|\Gr /H| \ge d_{\min} (\Gr) +1$.
Thus $d_{\min} (\Gr) \ge 1/\eps$ guaranties that $\Gr$ has no proper subgroups of index at most $1/\eps$. 
Another sufficient property for avoiding normal subgroups of index $1/\eps$ is 
simplicity of $\Gr$, of course.

\bigskip 

Finally, let us obtain an analogue of Corollary \ref{c:basis_ab}. 

\begin{corollary}
	Let $\Gr$ be a finite group, $d \ge 2$ be an integer and $B, \Omega \subseteq \Gr$, 
	be sets such that any element of $\Gr \setminus \Omega$ can be represented as a product of $d$ elements of $B B^{-1}$ or $B^{-1} B$ in at least $g\ge 1$ ways.
	Then 
	\begin{equation}\label{f:basis_nab}
	\la_1 (\Cay (B)) \ge \frac{g(|\Gr| - 2|\Omega|)}{8 d^2|B|^{2d}} \,.   
	\end{equation}
	If $|\Omega| = (1-\eps) |\Gr|$ and $\Gr$ has no normal proper subgroups of index at most $2/\eps$, then 
	\begin{equation}\label{f:basis_nab+}
	\la_1 (\Cay (B)) \ge \frac{\eps^{\log_{3/2} 3} g |\Gr|}{16d^2 |B|^{2d}}  \,.
	\end{equation}	
	\label{c:basis_nab}
\end{corollary}
\begin{proof} 
	Without loosing of the generality we consider the case $B B^{-1}$.
	Let $\d$ be as in formula \eqref{f:B/2} of Lemma \ref{l:B/2}.
	Then anyway one can take $\d = \frac{1}{2}$. 
	Also,  let 
	$P = \Bohr(\rho, \d)$ be a Bohr set with $\rho \neq 1$
	and let $P^c := \Gr \setminus P$.  
	By Lemma \ref{l:B/2} we know that $|P| \le |\Gr|/2$ and hence $|P^c| \ge |\Gr|/2$. 
	Since $(B * B^{-1})^{(d)} (x) \ge g$ for any $x\in \Gr \setminus \Omega$, we see that 
	\begin{equation}\label{f:sigma_basis_P^c_nab}
	\sigma^{(d)}_P (B*B^{-1}) = 1 - |B|^{-2d} \sigma^{(d)}_{P^c} (B*B^{-1}) \le 1 - \frac{g(|P^c|-|\Omega|)}{|B|^{2d}} 
		\le 1 - \frac{g |\Gr|/2 - g|\Omega|}{|B|^{2d}} \,. 
	\end{equation}
	Applying the first part of Theorem \ref{t:Lev_appl_non-abelian}
	with $\a = \frac{g |\Gr| - 2g|\Omega|}{2|B|^{2d}}$ and $\d$ as before, we derive
	\[
	\la_1 (\Cay (B)) \ge \frac{g(|\Gr| - 2|\Omega|)}{8 d^2|B|^{2d}}   
	\]
	as required. 
	To obtain \eqref{f:basis_nab+} use the first part of Theorem \ref{t:Lev_appl_non-abelian} with the parameters $\d=\d_{\eps/2} \ge (\eps/2)^{\log_{3/2} 2}$ and 
	$\a = \frac{\eps g |\Gr|}{2|B|^{2d}}$. 
	This completes the proof. 
	$\hfill\Box$
\end{proof}

\section{Examples}
\label{sec:examples}

Our first example of using the results from the previous Sections concerns maximal sets in non--abelian groups, avoiding non--affine equations.
For simplicity we consider just an equation with three variables.

\begin{corollary}
Let $\Gr$ be a finite group with the identity $e$ and $A\subseteq \Gr$ be a maximal set such that $e\notin A^3$. 
Then
\begin{equation}\label{f:A^3_1}
	\la_1 (\Cay(A)) \ge \frac{|\Gr|}{2(|A|+|\sqrt{A^{-1}}| + |{e}^{1/3}|)^2} - \frac{1+|\sqrt{A^{-1}}| + |{e}^{1/3}|}{|A|} \,,
\end{equation}
and 
\begin{equation}\label{f:A^3_2}
	\la_1 (\Cay(A \cup \sqrt{A^{-1}})) \ge \frac{|\Gr|}{2(|A|+ |{e}^{1/3}|)^2} - \frac{1+|e^{1/3}|}{|A|} \,.
\end{equation}
\label{c:A^3} 
\end{corollary}
\begin{proof} 
	Indeed, by maximality of $A$ we see that any $x\notin A$ either belongs to $A^{-1} A^{-1}$ or $x\in  \sqrt{A^{-1}}$ or $x\in e^{1/3}$. 
	In other words, the set  $(A^{}\cup \{e\})^{2}$ covers the group $\Gr$, excepting a set of size at most $|\sqrt{A^{-1}} \cup e^{1/3}|$
	and the set $(A^{}\cup \sqrt{A^{-1}} \cup \{e\})^{2}$ covers the group $\Gr$, excepting  a set of size at most $|e^{1/3}|$.
	Applying Theorem \ref{t:basis_g}, we obtain 
	\eqref{f:A^3_1}, \eqref{f:A^3_2}.
	This completes the proof. 
$\hfill\Box$
\end{proof}

\bigskip

In the next example we consider the family of so--called {\it $B_k$--sets}, see, e.g., \cite{O'Bryant}.
Recall that $A \subseteq \mathbb{N}$ is called a $B_k$--set, $k\ge 2$ if all sums $a_1+\dots+a_k$, $a_1, \dots, a_k \in A$ are distinct.

\begin{corollary}
	Let $A \subseteq \{1,2,\dots, N\}$ be a $B_k$--set and $N$ be a prime. 
	Suppose that $|A| \gg_k N^{1/k}$.
	Then there is a constant $c = c (k) >0$ such that for all  $r\neq 0$ one has 
	\begin{equation}
	\left| \sum_{x\in A} e^{2\pi i rx/N} \right| \le (1-c) |A| \,. 
	\end{equation}
\end{corollary}  
\begin{proof}
	Since $A$ is a $B_k$--set and $|A| \gg_k N^{1/k}$, it follows that there are 
	$$|kA| = \binom{|A|+k-1}{k} \gg_k |A|^k \gg_k N := \eps(k) N$$ elements of $kA$ belonging to  $\{1, \dots, kN \}$.
	Consider the set $A$ modulo $N$.
	Then modulo $N$ the set $kA \subseteq \Z/N\Z$ has at least $\eps(k) N/k$ elements. 
	Applying Corollary \ref{c:basis_ab} with $g=1$, $d=k$ and $|\Omega| = (1- \eps(k)/k)N$, we obtain 
	\[
	\la_1 (\Cay(A)) \gg_k \frac{N}{|A|^k} \gg_k  1 \,.
	\]
This completes the proof. 
$\hfill\Box$
\end{proof}

\bigskip 

Our third  example concerns the well--known problem of Erd\H{o}s and Tur\'an (see \cite{ET})
on the quantity 
$\limsup_n A^{(2)} (n)$ for an arbitrary basis $A\subseteq \mathbb{N}$ of order two. 
It was conjectured that the $\limsup$ equals infinity for any such $A$. 
We show that any 
basis of order two 
has  
a certain 
expansion property.

Given a set $A \subseteq \mathbb{N}$ denote by $A_N$ the intersection of $A$ with $\{1,\dots, N\}$.
Notice that if $\limsup_N |A_N|/N^{1/2} = \infty$, then, obviously, $\limsup_n A^{(2)} (n) =  \infty$. 

\begin{corollary}
	Let $A \subseteq \mathbb{N}$ be a set such that $A+A$ equals $\mathbb{N}$ up to  a finite number of exceptions. 
	Suppose that $|A_N| \le K N^{1/2}$ for all sufficiently large prime $N$. 
	Then there is a constant $c = c (K) >0$ such that for all sufficiently large $N$ and any $r\neq 0$ one has 
\begin{equation}
	\left| \sum_{x=1}^N A_N (x) e^{2\pi i rx/N} \right| \le (1-c) |A_N| \,. 
\end{equation}
\label{c:ET}	
\end{corollary}
\begin{proof}
	By the assumption there  is  a number $M$ such that $A^{(2)} (x) \ge 1$ for all $x\ge M$. 
	Take a  sufficiently large prime $N\ge 4M$ and consider the set $A_N := A \pmod N$. 
	Obviously, $2A_N$ contains at least 
	three quarters 
	of $\Z/N\Z$.
	Applying Corollary \ref{c:basis_ab} with $g=1$, $d=2$ and $|\Omega| \le N/4$, we obtain 
\[
	\la_1 (\Cay(A_N)) \gg \frac{N}{|A_N|^2} \ge \frac{1}{K^2} \,.
\]
This completes the proof. 
$\hfill\Box$
\end{proof}

\bigskip 

Corollary \ref{c:ET} shows in particular, that for a basis $A\subseteq \mathbb{N}$  the function $A^{(k)} (x)$ becomes more and more uniform as $k$ tends to infinity (see Corollary \ref{c:UD_basis}). 

\section{Appendix}
\label{sec:appendix}

In this Section we collect further natural properties of Bohr sets and connected notions, which have well--known abelian analogues.
We do this for the convenience  of the reader who is interested in this particular form of Bohr sets, most of these results are more or less contained in papers \cite{Bourgain_AP3}, \cite{Sanders_A(G)}, \cite{Sanders_doubling_metrics} and others.

In Section \ref{sec:non-abelian} we have  used the connection of the Bohr sets with the set of unitary representations $\rho$ such that  
$\| \FF{A} (\rho) \| \ge (1-\eps) |A|$ for a given 
set $A\subseteq \Gr$. 
Thus it is natural to give a more general

\begin{definition}
	Let $A\subseteq \Gr$ be a set, $\eps \in [0,1]$ be a real number.
	The {\it spectrum} $\Spec_\eps (A)$ of $A$ is the set unitary representations 
\[
	\Spec_\eps (A) = \{ \rho  ~:~ \| \FF{A} (\rho) \| \ge \eps |A| \} \,.
\]
\end{definition}

	Using the arguments of the proof of Lemma \ref{l:vlF}, we obtain a non--abelian  analogue  of the well--known result of Yudin \cite{Yudin}.

\begin{proposition}
	Let $A\subseteq \Gr$ be a set, and  $\eps_1, \eps_2 \in [0,1]$ be real numbers.
	Then 
\[
	\Spec_{1-\eps_1} (A) \cdot \Spec_{1-\eps_2} (A) \subseteq \Spec_{1-\eps_1 - \eps_2} (A) \,. 
\]
\end{proposition}
\begin{proof} 
	As we have from the arguments of the proof of Lemma \ref{l:vlF}, see estimate \eqref{tmp:01.04_1},  that  a unitary representation $\rho$ belongs to $\Spec_{1-\eps} (A)$ iff 
\[
	\| \sum_{g\in \Gr} (A * A^{-1}) (g) (I - \rho (g)) \| \le (2 \eps - \eps^2) |A|^2 = (1-(1-\eps)^2) |A|^2 \,.
\] 
	But 
\begin{equation}\label{f:1-xy}
	I-\rho_1 (g) \rho_2 (g) = (I-\rho_1 (g)) \rho_2 (g) + I - \rho_2 (g) 
\end{equation}
	and hence by the triangle inequality for the operator norm, we get 
\[
	\| \sum_{g\in \Gr} (A * A^{-1}) (g) (I - \rho_1 (g) \rho_2 (g)) \| 
		\le 
		(2 \eps_1 - \eps^2_1 + 2 \eps_2 - \eps^2_2) |A|^2 
			=
			(1-(1- \eps_1-\eps_2)^2) |A|^2
\]
	as required. 
$\hfill\Box$
\end{proof}

\bigskip 

Our 
next 
result 
shows that $\Bohr(\rho, \delta)$ has small product and hence it is possible to check the condition of smallness of the quantity $\sigma^{(d)}_P (B) \le 1-\alpha$ from  Theorem \ref{t:Lev_appl_non-abelian} just for sets with small product.

\begin{proposition}
	Let $\delta \in [0,2/5]$ be a real number and $\rho$ be a unitary representation. 
	Then
$$
	|\Bohr(\rho, \delta) \cdot \Bohr(\rho, \delta)| \le 2^{\frac{21 d_\rho^2}{2}}  |\Bohr(\rho, \delta)| 
$$
	and  there are sets $X, Y \subseteq \Gr$, $|X|, |Y| < 2^{25 d_\rho^2}$ such that 
\[
	\Bohr(\rho, \delta) \subseteq \Bohr(\rho, \delta/2) X\,, \quad \quad \Bohr(\rho, \delta) \subseteq Y \Bohr(\rho, \delta/2) \,.
\]
\label{p:Bohr_doubling}
\end{proposition}
\begin{proof}
	Write $k=d_\rho$. 
	In view of \eqref{f:Bohr_sums} it is enough to compare sizes of $|\Bohr(\rho, \delta)|$ and $|\Bohr(\rho, 2\delta)|$. 
	Further, one can check that $2(1-\cos \theta) \le \theta^2$ and $2(1-\cos \theta) \ge \theta^2/2$ for $|\theta|\le \sqrt{6}$. 
	Put
\[
	\eta := \eta (\delta) = \frac{1}{2\pi} \arccos \left(1 - \frac{\delta^2}{2} \right) \,.
\] 
	We have 
$
	\frac{\delta}{2\pi} \le \eta(\delta) \le \frac{\delta}{\pi}  
$.
	Let $U(\delta)$ be the set of the unitary matrices $U$ such that $\| U - I\| \le \delta$. 
	In \cite[Lemma 17.4]{Sanders_A(G)} it was proved that the Haar measure $\mu$ of $U(\delta)$ equals
\[
	\mu (U(\delta))
	=
	\frac{1}{k!} \int_{-\eta}^{\eta} \dots \int_{-\eta}^{\eta}\, \prod_{1\le n<m \le k} |e^{2\pi i \theta_n} - e^{2\pi i \theta_m}|^2  \,d\theta_1 \dots d\theta_k 
		=
\]
\begin{equation}\label{tmp:03.04_1}
		=
	\frac{1}{k!} \int_{-\eta}^{\eta} \dots \int_{-\eta}^{\eta}\, \prod_{1\le n<m \le k} 2(1-\cos(2\pi (\theta_n-\theta_m))) \,d\theta_1 \dots d\theta_k \,.  
\end{equation}
	Put 
\[
	F(k) = \frac{(2\pi)^{2\binom{k}{2}}}{k!} \int_{-1}^{1} \dots \int_{-1}^{1}\, \prod_{1\le n<m \le k} (\theta_n-\theta_m)^2 \,d\theta_1 \dots d\theta_k \,.  
\]
	From \eqref{tmp:03.04_1} and our bounds for $2(1-\cos \theta)$,  it follows that 
\[
	 2^{-\binom{k}{2}} 	\eta^{k^2} F(k) \le \mu (U(\delta)) \le \eta^{k^2} F(k) 
\]
	because $4\pi \eta(1) =2\pi/3 \le \sqrt{6}$. 
	Using the assumption $\delta \le 2/5$ and the previous formula, we obtain 
\begin{equation}\label{tmp:18.04_2}
	\frac{\mu (U(5\delta/2))}{\mu (U(\delta/2))} \le 2^{\binom{k}{2}} \cdot \frac{\eta(5\delta/2)^{k^2}}{\eta(\delta/2)^{k^2}}
		\le
			2^{\binom{k}{2}} 2^{10k^2} < 2^{\frac{21k^2}{2}} \,.
\end{equation}
	Now let $V = \rho (\Gr)$.
	It is easy to see that for any unitary matrix $u$ one has $|\Bohr(\rho, \delta)| \ge |V\cap U(\d/2) u|$
	because $U(\d/2) u (U(\d/2) u)^{-1}  \subseteq U(\delta)$. 
	Further, integrating over the Haar measure, we get in view of \eqref{tmp:18.04_2} 
\[
	|\Bohr(\rho, 2\delta)| = (\mu (U(\d/2)) )^{-1} \int_{} |\rho (\Bohr(\rho, 2\delta)) \cap U(\d/2) u| \,d u  
		\le
\]
\[ 
		\le 
		|\Bohr(\rho, \delta)| \frac{\mu (U(5\d/2))}{\mu (U(\d/2))} < 2^{\frac{21k^2}{2}} |\Bohr(\rho, \delta)| \,.
\]

	Now by the Ruzsa covering lemma (see, e.g., \cite{TV}) one finds $X$ (and similarly $Y$) such that 
\[
	\Bohr(\rho, \delta) \subseteq  \Bohr(\rho, \delta/4) \cdot \Bohr^{-1}(\rho, \delta/4) \cdot X \subseteq  \Bohr(\rho, \delta/2) \cdot X \,, 
\]
where
as above 
\[
	|X| \le \frac{|\Bohr(\rho, 5\delta/4)|}{|\Bohr(\rho, \delta/4)|} 
		\le
		\frac{\mu (\Bohr(\rho, 11\delta/8))}{\mu (\Bohr(\rho, \delta/8))}
		\le  2^{\binom{k}{2}} \frac{\eta(11\delta/8)^{k^2}}{\eta(\delta/8)^{k^2}} < 2^{25 k^2} \,.
\]
	This completes the proof. 
$\hfill\Box$
\end{proof}

\bigskip

Having a lower bound for size of one--dimensional Bohr set (see \cite[Lemma 17.3]{Sanders_A(G)} or Proposition above) one can obtain a lower bound for size of Bohr sets with an arbitrary $\G$.

\begin{proposition}
	Let $\Bohr (\rho_j, \delta_j)$, $j=1,\dots, k$ be Bohr sets such that  $\d_1 \le \d_2 \le \dots \le \d_k$.   
	Then 
\[
	|\Bohr (\{\rho_1, \dots, \rho_k \}, \delta_k)| \ge |\Gr|^{-1} \prod_{j=1}^k  |\Bohr (\rho_j, \delta_j/2)| \,. 
\]
\end{proposition}
\begin{proof} 
	Let $B = \Bohr (\{\rho_1, \dots, \rho_k \}, \delta_k)$ and $B_j = \Bohr (\rho_j, \delta_j/2)$, $j=1,2,\dots, k$.  
	Clearly, for any $j$ one has $B_j - B_j \subseteq B$. 
	Hence 
\begin{equation}\label{tmp:20.04_2}
	\sigma:= \sum_{x\in \Gr} (B_1 * B^{-1}_1) (x) \dots (B_k * B^{-1}_k) (x) =  \sum_{x\in B} (B_1 * B^{-1}_1) (x) \dots (B_k * B^{-1}_k) (x)
		\le
			|B| |B_1| \dots |B_k| \,. 
\end{equation}
	On the other hand, in view of formulae  \eqref{f:Parseval_representations}, \eqref{f:convolution_representations}, we get 
\begin{equation}\label{tmp:20.04_3}
	\sigma = \frac{1}{|\Gr|} \sum_{\rho} d_\rho \langle \FF{B}_1 (\rho) \FF{B}^*_1 (\rho) \dots \FF{B}_{k-1} (\rho) \FF{B}^*_{k-1} (\rho),  
		\FF{B}_k (\rho) \FF{B}^*_k (\rho) \rangle 
		\ge 
		\frac{|B_1|^2 \dots |B_k|^2}{|\Gr|} 
\end{equation}
	because the operators  $\FF{B}_1 (\rho) \FF{B}^*_1 (\rho)$ are hermitian and non--negatively defined. 
	Comparing \eqref{tmp:20.04_2}, \eqref{tmp:20.04_3}, we obtain the result. 
$\hfill\Box$
\end{proof}

\bigskip 

A Bohr set $\Bohr (\rho, \delta)$ is called to be {\it regular} if 
\[
	\left| |\Bohr (\rho, (1+\kappa)\delta)| -  |\Bohr (\rho, \delta)| \right| \le 100 d^2_\rho |\kappa| \cdot |\Bohr (\rho, \delta)| \,, 
\] 
whenever $|\kappa| \le 1/(100 d^2_\rho)$. 
Even in the abelian case it is easy to see  that not each Bohr set is regular. 
Nevertheless, it was showed in \cite{Bourgain_AP3} that  for $\Gr = \Z/N\Z$ one can find a regular Bohr set decreasing the parameter $\delta$ slightly. 
We show the same for general groups, repeating the arguments from  \cite[Lemma 4.25]{TV} (also, see \cite[Lemma 9.3]{Sanders_A(G)}).

\begin{proposition}
	Let $\delta \in [0,1/2]$ be a real number and $\rho$ be a unitary representation. 
	Then there is $\delta_1 \in [\delta,2\delta]$ such that $\Bohr (\rho, \delta_1)$ is regular. 
\end{proposition}
\begin{proof}
	Consider the non--decreasing function $f : [0,1] \to \mathbb{R}$ defined as 
	$$
		f(a) := d^{-2}_\rho \log \mu(\Bohr (\rho, 2^a \delta)) \,.
	$$
	By the first part of Proposition \ref{p:Bohr_doubling}, we have $f(1)- f(0) \le \log (21/2)$. 
	Clearly, if we could find $a\in [0.1, 0.9]$ such that $|f(a) - f(a')| \le 25|a-a'|$ for all $|a'-a| \le 0.1$, 
	then the set $\Bohr (\rho, 2^a \delta)$ is regular.
	If not, then for every such $a$ 
	there is 
	an interval $I_a$,  $a\in I_a$, $|I_a| \le 0.1$ with $\int_{I_a} df > 25 |I_a|$. 
	Obviously, these intervals cover   $[0.1, 0.9]$ and by the Vitali covering lemma one can find a finite subcollection of disjoint intervals of total measure at least $0.8/5$, say. 
	But then 
	$$
		\log (21/2) \ge \int_{0}^{1} df \ge 25 \cdot 0.8/5 = 4
	$$
	and this is a contradiction. 
$\hfill\Box$
\end{proof}


\bigskip 

Finally, let us say something non--trivial about the spectrum of regular Bohr sets.

\begin{proposition}
	Let $B=\Bohr(\rho, \delta)$ be a regular Bohr set, and $B' = \Bohr(\rho, \delta')$, where  $\delta' \le \kappa \delta/(100 d_\rho^2)$ and 
	$\kappa \in (0,1)$ be a real number. 
	Then 
$$
	\Spec_{\eps} (B) \subseteq \Spec_{1-\frac{2\kappa}{\eps}} (B') \,.
$$
\end{proposition} 
\begin{proof}
	Let $\pi \in \Spec_{\eps} (B)$.
	Also, let $B^{+} = \Bohr(\rho, \delta+\delta')$, $B^{-} = \Bohr(\rho, \delta-\delta')$. 
	We have 
\[
	\eps |B| \le \| \FF{B} (\pi) \| \le  |B'|^{-1} \| \FF{B} (\pi)\| \| \FF{B}' (\pi)\| +
	\| \sum_x (B(x) - |B'|^{-1} (B*B')(x)) \pi (x) \| 
	\le
\]
\begin{equation}\label{tmp:20.04_1}
	\le
	 |B'|^{-1} \| \FF{B} (\pi)\| \| \FF{B}' (\pi)\| + \sum_x |B(x) - |B'|^{-1} (B*B')(x)| \,.
\end{equation}
	It is easy to see that the summation in \eqref{tmp:20.04_1} is taken over $B^{+}\setminus B^{-}$. 
	By the regularity of $B$ one can estimate this sum as  $2\kappa |B|$.
	Hence
\[
	\| \FF{B}' (\pi)\|  \ge |B'| (1-2\kappa |B|\| \FF{B} (\pi)\|^{-1}) \ge |B'| (1-2\kappa \eps^{-1}) 
\]
	or, in other words, $\pi \in \Spec_{1-2\kappa\eps^{-1}} (B')$.  
	This completes the proof. 
$\hfill\Box$
\end{proof}


\bigskip

\noindent{I.D.~Shkredov\\
Steklov Mathematical Institute,\\
ul. Gubkina, 8, Moscow, Russia, 119991}
\\
and
\\
IITP RAS,  \\
Bolshoy Karetny per. 19, Moscow, Russia, 127994\\
and 
\\
MIPT, \\ 
Institutskii per. 9, Dolgoprudnii, Russia, 141701\\
{\tt ilya.shkredov@gmail.com}

\end{document}